\documentclass[11pt,a4paper]{amsart}
\usepackage{amsmath}
\usepackage{amsfonts}
\usepackage{graphicx}
\usepackage{framed}
\usepackage{amssymb,cancel}
\usepackage{xcolor}

\usepackage{etex}
\usepackage{latexsym}
\usepackage[english]{babel}
\usepackage[utf8x]{inputenc}

\def \R {\mathbb{R}}

\def \de {\partial}

\def \e {\varepsilon}

\def \LL {\mathcal{L}}
\def \xp {\mathbb{X}_p(\Omega)}
\usepackage{amsthm}
\theoremstyle{definition}
\newtheorem{definition}{Definition}[section]

\newtheorem{remark}[definition]{Remark}

\theoremstyle{plain}
\newtheorem{theorem}[definition]{Theorem}
\newtheorem{proposition}[definition]{Proposition}
\newtheorem{lemma}[definition]{Lemma}
\newtheorem{corollary}[definition]{Corollary}

\numberwithin{equation}{section}
\begin{document}

 \title[A Brezis-Oswald approach for mixed operators]{A Brezis-Oswald approach for mixed local and nonlocal operators}
 \author[S.\,Biagi]{Stefano Biagi}
 \author[D.\,Mugnai]{Dimitri Mugnai}
 \author[E.\,Vecchi]{Eugenio Vecchi}
 
 \address[S.\,Biagi]{Dipartimento di Matematica
 \newline\indent Politecnico di Milano \newline\indent
 Via Bonardi 9, 20133 Milano, Italy}
 \email{stefano.biagi@polimi.it}

 \address[D.\,Mugnai]{Dipartimento di Ecologia e Biologia (DEB)
 \newline\indent Università della Tuscia \newline\indent
 Largo dell'Università, 01100 Viterbo, Italy}
 \email{dimitri.mugnai@unitus.it}
 
 \address[E.\,Vecchi]{Dipartimento di Matematica
 \newline\indent Università degli Studi di Bologna \newline\indent
 Piazza di Porta San Donato 5, 40126 Bologna , Italy}
 \email{eugenio.vecchi2@unibo.it}



\thanks{The authors are members of the {\em Gruppo Nazionale per
l'Analisi Ma\-te\-ma\-ti\-ca, la Probabilit\`a e le loro Applicazioni}
(GNAMPA) of the {\em Istituto Nazionale di Alta Matematica} (INdAM). S. Biagi
has been partially supported by the INdAM-GNAMPA project 
\emph{Metodi topologici per problemi al contorno associati a certe 
classi di equazioni alle derivate parziali}.
D. Mugnai has been supported by the FFABR ``Fondo per il finanziamento delle attivit\`a base di ricerca'' 2017 and by the INdAM-GNAMPA Project {\em Equazioni alle derivate parziali: problemi e mo\-del\-li}.
E. Vecchi has been partially supported
by the INdAM-GNAMPA project 
\emph{Convergenze variazionali per funzionali e operatori dipendenti da campi vettoriali}.}
 \begin{abstract}
  In this paper we provide necessary and sufficient conditions
  for the existence of a unique positive weak
  solution for some sublinear Dirichlet problems 
  dri\-ven by the sum of a quasilinear local and a nonlocal operator, i.e., 
  $$\LL_{p,s} = -\Delta_p + (-\Delta)^s_p.$$
  Our main result is resemblant to the celebrated work 
  by Brezis--Oswald \cite{BO}. In addition, we prove a regularity result of independent interest.
 \end{abstract}
 \maketitle 
\section{Introduction}\label{sec.Intro}
In this paper we are concerned with quasilinear problems driven 
by the sum of a local and a nonlocal operator. More precisely, the leading operator is 
\[
\LL_{p,s} u := -\Delta_p u +(-\Delta)_p^s u.
\]
Here, $\Delta_p u = \mathrm{div}(|\nabla u|^{p-2}\,\nabla u)$ is the classical $p-$Laplacian operator and, 
for fixed $s\in (0,1)$ and up to a multiplicative positive constant, the fractional $p-$Laplacian is 
defined as
  $$(-\Delta)_p^s u(x) := 2\,\mbox{P.V.}\int_{\R^n}
  \frac{|u(x)-u(y)|^{p-2}(u(x)-u(y))}{|x-y|^{n+ps}}\, dy,$$
where P.V. denotes the Cauchy principal vale, namely
\[
\begin{aligned}
&\mbox{P.V.}\int_{\R^n}
  \frac{|u(x)-u(y)|^{p-2}(u(x)-u(y))}{|x-y|^{n+ps}}\, dy\\
  &=\lim_{\epsilon\to0}\int_{\{y\in \R^n\,:\,|y-x|\geq \epsilon\}}
  \frac{|u(x)-u(y)|^{p-2}(u(x)-u(y))}{|x-y|^{n+ps}}\, dy.
\end{aligned}
\]
Problems driven by operators like $\LL_{p,s}$ have raised a certain interest in the last few years, both 
for the mathematical complications that the combination of two so different operators imply and for the 
wide range of applications, see for instance \cite{BDVV, BDVVAcc, BDVVperp, BSM, SilvaSalort, DPLV1, DPLV2} 
and the references therein. A common feature of the a\-fo\-re\-men\-tio\-ned 
papers is to deal with {\it weak 
solutions}, in contrast with other results existing in the literature where viscosity solutions have been 
considered, see e.g. \cite{Barles1, Barles2}.

The purpose of this paper is to prove an existence and uniqueness result in the spirit of the celebrated 
paper by Brezis-Oswald for the Laplacian, see \cite{BO}. So, let us consider the Dirichlet problem
 \begin{equation} \label{eq.pbDirSec2}
   \begin{cases} 
 -\Delta_p u +(-\Delta)_p^s u = f(x,u) & \text{in $\Omega$}, \\
    u \gneqq 0 & \text{in $\Omega$}, \\
    u \equiv 0 & \text{in $\R^n\setminus \Omega$}.
   \end{cases}
  \end{equation}
Here $\Omega$ is a bounded open set with $C^1$-smooth boundary.
Under standard assumptions on $f$, we show that if $u$ solves \eqref{eq.pbDirSec2} (in some sense to be 
made precise later on), then $u>0$ in $\Omega$, and we give precise conditions under which such a solution 
exists and is unique. For this, as in \cite{BO} for the local case with $p=2$, a crucial role is played by 
the monotonicity of the map
\[
t\mapsto \frac{f(x,t)}{t^{p-1}}.
\]
Indeed, in \cite{BO} the authors considered the problem
\begin{equation}\label{Lapla}
\begin{cases}
-\Delta u=f(x,u) & \mbox{ in }\Omega,\\
u\gneqq 0, & \mbox{ in }\Omega \\
u=0 &\mbox{ on }\partial \Omega.
\end{cases}
\end{equation}
 where $f:\Omega\times [0,\infty)\to \mathbb{R}$ satisfies suitable growth assumptions, and
 the map
 \[
t\mapsto \frac{f(x,t)}{t}\]
 is decreasing in $(0,\infty)$. Under these conditions, 
 the authors showed that \eqref{Lapla} has at most one solution and 
 that such a solution exists if and only if
\begin{equation}\label{autov0}
\lambda_1(-\Delta -\tilde{a}_0(x))<0
\end{equation}
and
\begin{equation}\label{autovinfty}
\lambda_1(-\Delta -\tilde{a}_\infty(x))>0,
\end{equation}
where $\lambda_1(-\Delta -a(x))$  denotes the first eigenvalue of $-\Delta -a(x)$ with zero Dirichlet 
condition and
\begin{equation}\label{defABrezis}
\tilde{a}_0(x):=\lim_{u\downarrow 0}\frac{f(x,u)}{u}  
\quad \mbox{ and }\quad \tilde{a}_\infty(x):=\lim_{u\uparrow \infty}\frac{f(x,u)}{u}.
\end{equation}
As already mentioned, in this paper we want to prove an analogous result in the quasilinear case given by 
problem \eqref{eq.pbDirSec2}, where $f$ satisfies the following conditions:
  \begin{itemize}
   \item[(f1)] $f:\Omega\times [0,+\infty)\to \mathbb{R}$ is a Carath\'eodory function.
   \item[(f2)]    $f(\cdot,t)\in L^{\infty}(\Omega)$ for every $t\geq 0$.
   \item[(f3)] There exists a constant $c_p > 0$ such that
  \begin{equation*}
	|f(x,t)|\leq c_p(1+t^{p-1})\qquad \text{for a.e.\,$x\in\Omega$ and every $t\geq 0$}.
  \end{equation*}
   \item[(f4)] For a.e.\,$x\in\Omega$, the function 
   $t\mapsto \dfrac{f(x,t)}{t^{p-1}}$ is strictly decreasing in $(0,\infty)$.
 \end{itemize}
We can then consider functions $a_0$ and $a_{\infty}$ akin to those in \eqref{defABrezis}, see 
\eqref{a0ainfty} for the precise definition. Moreover, we denote respectively by 
\begin{equation} \label{eq:lambda1introdef}
 \lambda_1(\mathcal{L}_{p,s}-a_0)\qquad\text{and}\qquad 
 \lambda_1(\mathcal{L}_{p,s}-a_\infty),
\end{equation}
the smallest eigenvalues of $\mathcal{L}_{p,s}-a_0$ and $\mathcal{L}_{p,s}-a_\infty$, both in presence of 
nonlocal Dirichlet boundary condition (i.e. $u=0$ in $\mathbb{R}^n \setminus \Omega$). Since the function 
$a_0$  can be \emph{unbounded}, (notice that, on the other hand, $a_\infty$ is bounded by (f3)),
similarly to Brezis-Oswald \cite{BO},
the precise de\-fi\-ni\-tion of \eqref{eq:lambda1introdef} 
is the following:
\begin{equation} \label{eq:deflambdaBO}
 \begin{split}
  &  \lambda_1(\LL_{p,s}-a_0)
  := \!\!\!\inf_{\begin{subarray}{c}
 u\in \xp \\
 \|u\|_{L^p(\Omega)} = 1
 \end{subarray}}\bigg\{\mathcal{Q}_{p,s}(u)-\int_{\{u\neq 0\}}a_0\,|u|^p\, dx \bigg\};
 \\[0.2cm]
 & \lambda_1(\LL_{p,s}-a_\infty)
  := \!\!\!\inf_{\begin{subarray}{c}
 u\in \xp \\
 \|u\|_{L^p(\Omega)} = 1
 \end{subarray}}\bigg\{\mathcal{Q}_{p,s}(u)-\int_{{\Omega}}a_\infty\,|u|^p\, dx \bigg\},
 \end{split}
\end{equation}
 where $\mathbb{X}_p(\Omega)$ is defined in \eqref{eq.defSpaceXp}, and we have introduced the simplified notation
 \begin{equation} \label{eq:defQform}
  \mathcal{Q}_{p,s}(u) := \int_{\Omega}|\nabla u|^p\, dx
 + \iint_{\R^{2n}}\!\!\!
 \frac{|u(x)-u(y)|^p}{|x-y|^{n+ps}}\, dx\, dy.
 \end{equation}
In order to prove \emph{uniqueness}, we shall add the following additional hypothesis:
 \begin{itemize}
  \item[(f5)] there exists $\rho_f > 0$ such that
  \begin{equation} \label{eq.ineqrhof}
   f(x,t) > 0 \quad\text{for a.e.\,$x\in\Omega$ and every $0<t<\rho_f$}.
  \end{equation}
 \end{itemize}
  We observe that, in the particular case of power-type
  nonli\-ne\-a\-ri\-ties $f(x,u) = u^\theta$ (with $0\leq \theta \leq p-1$), assumption
  (f5) is trivially satisfied.
 \begin{remark} \label{rem:assumptionf5}
 As a matter of fact, assumption (f5) is just a {technical} one
 (far from being optimal) which permits to overcome the lack of 
 \emph{boun\-da\-ry regularity} for $\LL_{p,s}$. In fact,
  since we do not know the $C^{1,\alpha}$-regularity {up to the boundary} 
  of weak solutions
  of \eqref{eq.pbDirSec2}, we do not have at our disposal a Hopf-type lemma
  for $\LL_{p,s}$ and we cannot follow directly the approach by Brezis-Oswald
  to get
  the {uniqueness of solutions} for \eqref{eq.pbDirSec2}. 
  We then need to exploit
  a suitable approximation argument (see Theorem \ref{thm.uniqueness} below), 
  and for this approach assumption (f5) seems
  to be essential.
\end{remark}
We are now ready to state our main result:
\begin{theorem} \label{thm:Main}
Let $\Omega \subset \mathbb{R}^n$ be a bounded open set with $C^1$-smooth boundary $\partial\Omega$. 
Assume that $f$ satisfies \emph{(f1)}--\emph{(f5)}. Then, the following assertions hold.
\begin{itemize}
\item[(1)] There exists a unique positive solution to \eqref{eq.pbDirSec2} if
 $$\lambda_1(\mathcal{L}_{p,s}-a_0)<0<\lambda_1(\mathcal{L}_{p,s}-a_\infty).$$
 Moreover, if a solution to \eqref{eq.pbDirSec2} exists, then it is unique and 
 \begin{equation*}
  \lambda_1(\mathcal{L}_{p,s}-a_0)<0.
  \end{equation*}
\item[(2)] In the linear case $p = 2$, there exists a unique 
positive solution to \eqref{eq.pbDirSec2}
 \emph{if and only if}
 \begin{equation*}
  \lambda_1(\mathcal{L}_{2,s}-a_0)<0<\lambda_1(\mathcal{L}_{2,s}-a_\infty).
  \end{equation*}
\end{itemize}
\end{theorem}
Let us now spend a few comments on the proof of Theorem \ref{thm:Main}. Despite the apparent simplicity in 
working with operators like $\mathcal{L}_{p,s}$, we have to face some difficulties related to the scarce 
literature available for such operators.
First of all, 
we need to prove the validity of the {\sl strong maximum principle} as stated in \cite{PS}, namely: 
if $u$ is a \emph{nonnegative} weak solution of
 $\LL_{p,s}u = f(x,u)$ (with zero-boun\-da\-ry conditions), then
\[
\text{either $u\equiv 0$ in $\Omega$ \quad or  \quad $u > 0$ a.e.\,in $\Omega$},
\]
see Theorem \ref{thm.SMPWeak} for the precise statement. We believe that this preliminary result is of 
independent interest, and we stress that Theorem \ref{thm.SMPWeak} cannot be deduced as a corollary of the 
maximum principles proved in \cite{BDVV} nor in \cite{BSM}.

A second delicate point concerns the uniqueness of the solution. Indeed, the lack of even a basic boundary 
regularity for $\mathcal{L}_{p,s}$ prevents from applying the original argument in \cite{BO}. For this 
reason, we have to exploit an approximation argument inspired by the one in \cite{BS}, with the additional 
aid of a further assumption on $f$ (see \eqref{eq.ineqrhof}).

Finally, we emphasize that in order to get a complete characterization of the existence and uniqueness of a po\-si\-ti\-ve weak solution, we must restrict ourselves to the linear case $p=2$, see Proposition \ref{PropAinfPos}, since we cannot prove the inequality $\lambda_1(\mathcal{L}_{p,s}-a_\infty)>0$ in the general case. Indeed, two fundamental tools would be needed to prove this fact: first, an $L^\infty$ bound on solutions, and this is the content of Theorem \ref{thm:uzeroglobalBd}; second, some nonlinear Green identities, used  in \cite{DiazSaa} and in \cite{FMP} for the local case. To the best of our knowledge, the nonlocal counterparts of such identities are still missing in the literature. Nevertheless, we think that the global  boundedness result in Theorem \ref{thm:uzeroglobalBd}, as well as the very recent results in \cite{Kinnunen}, can be a useful tool for further investigations in the general case $p\neq 2$. 

We conclude by noticing that, in the purely nonlocal case, the complete characterization is possible for any $p$ since the precise behaviour of the solutions at the boundary is known, see \cite{im}.
Actually, even if an analogous result
is not known in our mixed context, after the submission of this paper we found a way to bypass
both the absence of appropriate nonlinear Green identities and the lack of
boundary regularity for $\LL_{p,s}$; thus, we can go full circle
and obtain a complete characterization of the (unique) solvability of
\eqref{eq.pbDirSec2} for $p\neq 2$ as well, see \cite{BMV}.
\medskip

We close this introduction with a plan of the paper: in Section \ref{sec.NotPrel} we introduce the relevant 
notation and we list the standing assumptions needed in the rest of the paper. Then, in Section \ref{sec:SMP} 
we prove the strong maximum principle for weak solutions of problem \eqref{eq.pbDirSec2}. Uniqueness and boundedness
of positive solutions is proved in Section \ref{sec.uniqueness}.
In order to prove conditions analogous to those established in \eqref{autov0} and \eqref{autovinfty}, in 
Section \ref{sec.Eigenvalue} we shall study the eigenvalue problem associated to $\LL_{p,s}$ in presence of 
a bounded and indefinite weight. In fact, although the analogue of the functions defined in 
\eqref{defABrezis} could be unbounded, for the existence-uniqueness result we will reduce to study an 
eigenvalue problem in presence of a bounded weight, see Proposition \ref{PropAinfPos}. Finally, existence is proved in Section
\ref{sec.Existence}.

\medskip
\textbf{Acknowledgements} We thank the anonymous referee for his/her careful reading of the manuscript.

\section{Notation and preliminary results} \label{sec.NotPrel}
  In this first section, we introduce the main assumptions and notation
  which shall be used throughout the rest of the paper. Moreover,
  we state and prove some auxiliary results which shall be exploited
  in the next sections.
  \medskip
  
  To begin with, we fix $p\in(1,+\infty)$ and we let $\Omega\subseteq\R^n$
  be a connected and bounded open set with $C^1$-smooth
  boundary $\de\Omega$.
  Accordingly, we define
  \begin{equation} \label{eq.defSpaceXp}
   \mathbb{X}_p(\Omega) := \big\{u\in W^{1,p}(\R^n):\,
   \text{$u\equiv 0$ a.e.\,on $\R^n\setminus\Omega$}\big\}.
  \end{equation}
  In view of the regularity assumption on $\de\Omega$, it is well-known that (see e.g. \cite[Proposition 9.18]{Brezis})
  $\mathbb{X}_p(\Omega)$ can be identified
  with the space $W_0^{1,p}(\Omega)$: more precisely, we have
  \begin{equation} \label{eq.identifXWzero}
   u \in W_0^{1,p}(\Omega)\,\,\Longleftrightarrow\,\,
  u\cdot\mathbf{1}_{\Omega}\in\mathbb{X}_p(\Omega),
  \end{equation}
  where $\mathbf{1}_\Omega$ is the indicator function of $\Omega$.
  From now on, we shall \emph{tacitly i\-den\-ti\-fy}
  a function $u\in W_0^{1,p}(\Omega)$ with its `zero-extension' 
  $\hat{u} := u\cdot\mathbf{1}_\Omega\in\mathbb{X}_p(\Omega)$.
  \vspace*{0.1cm}
  
By the Poincar\'e inequality and \eqref{eq.identifXWzero}, we get that the quantity
  $$\|u\|_{\mathbb{X}_p} :=\left( \int_{\Omega}|\nabla u|^p\, dx\right)^{1/p},
  \qquad u\in\mathbb{X}_p(\Omega),$$ 
  endows $\mathbb{X}_p(\Omega)$ with a structure of (real) Banach space, which is
  actually isometric to $W_0^{1,p}(\Omega)$. In par\-ti\-cu\-lar, the following 
  properties hold true:
  \medskip
  \begin{enumerate}
   \item $\mathbb{X}_p(\Omega)$ is separable and reflexive (since $p > 1$);
   \vspace*{0.05cm}
   
   \item $C_0^\infty(\Omega)$ is dense in $\mathbb{X}_p(\Omega)$.
  \end{enumerate}
  Due to its relevance in the sequel, we also introduce an \emph{ad-hoc}
  notation for the (con\-vex) cone
  of the nonnegative functions in $\mathbb{X}_p(\Omega)$:
  \[
   \mathbb{X}_p^+(\Omega) := \big\{u\in\mathbb{X}_p(\Omega):\,\text{$u\geq 0$ a.e.\,in $\Omega$}\big\}.
  \]
  As anticipated in the Introduction, the aim of this paper is to provide
  necessary and sufficient conditions for solving the Dirichlet problem \eqref{eq.pbDirSec2}. 
 \medskip
 
First of all, we give the definition of ``solution" for \eqref{eq.pbDirSec2}.
 \begin{definition} \label{def.WeakSol}
 Let the above assumptions and notation be in force. We say that a function
 $u\in\mathbb{X}_p(\Omega)$ is a \emph{weak solution} of 
 \eqref{eq.pbDirSec2} if
 \begin{enumerate}
  \item for every function $\varphi\in\mathbb{X}_p(\Omega)$ one has
  \begin{equation} \label{eq.defWeakSol}
  \begin{split}
    & \int_{\Omega}
   |\nabla u|^{p-2}\langle \nabla u,\nabla \varphi\rangle\, dx \\
   & \qquad\qquad + \iint_{\R^{2n}}
   \!\!\!\frac{|u(x)-u(y)|^{p-2}(u(x)-u(y))(\varphi(x)-\varphi(y))}{|x-y|^{n+ps}}\, dx\, dy
   \\
   &   = \int_{\Omega}f(x,u)\varphi\, dx;
   \end{split}
  \end{equation}
  \item $u\geq 0$ a.e.\,in $\Omega$ and $|\{x\in\Omega:\,u(x) > 0\}| > 0$, $|A|$ 
  denoting the Lebesgue measure of the set $A$.
 \end{enumerate}
 \end{definition}
 \begin{remark} \label{rem.BuonaDef}
 A couple of remarks on Definition \ref{def.WeakSol} are in order. 
 \medskip
 
 \noindent(1)\,\,We explicitly notice that the definition above is well-posed. 
  Indeed, we know from \cite[Proposition 2.2]{guida}  that there exists $c_{n,s,p}>0$ such that 
 $$
 \bigg(\iint_{\R^{2n}}
   \!\!\!\frac{|u(x)-u(y)|^{p}}{|x-y|^{n+ps}}\, dx\, dy\bigg)^{1/p}
   \leq c_{n,s,p}\|f\|_{W^{1,p}(\R^n)}
  \quad\text{$\forall\,\,f\in W^{1,p}(\R^n)$}.
$$
 Thus, by using H\"older's inequality, we find that  if $u, \varphi\in \mathbb{X}_p(\Omega)$, then
 \begin{align*}
   & 
   \iint_{\R^{2n}}
   \!\!\!\frac{|u(x)-u(y)|^{p-1}|\varphi(x)-\varphi(y)|}{|x-y|^{n+ps}}\, dx\, dy
   \\
   & \quad
   \leq 
   \bigg(\iint_{\R^{2n}}
   \!\!\!\frac{|u(x)-u(y)|^{p}}{|x-y|^{n+ps}}\, dx\, dy\bigg)^{1-1/p}
   \bigg(\iint_{\R^{2n}}
   \!\!\!\frac{|\varphi(x)-\varphi(y)|^{p}}{|x-y|^{n+ps}}\, dx\, dy\bigg)^{1/p} \\[0.15cm]
   & \quad \leq c_{n,s,p}^2\|u\|_{W^{1,p}(\R^n)}^{p-1}\cdot\|\varphi\|_{W^{1,p}(\R^n)} < +\infty.
  \end{align*}
 On the other hand, since $u\in\mathbb{X}_p(\Omega)\subseteq L^p(\Omega)$, 
 by exploiting (f3), we have
  $$\int_\Omega|f(x,u)\,\varphi|\, dx
  \leq c_p\left(\int_\Omega|\varphi| dx+\int_\Omega|u|^{p-1}|\varphi|\, dx\right)
  <+\infty.$$
  
  \noindent(2)\,\,It is clear from (1) that the notion of solution similarly holds if we replace
  the growth assumption in (f3) with the more general one
 \begin{equation}\label{cresc}
   |f(x,t)|\leq c_p(1+t^{q-1})\qquad \text{for a.e.\,$x\in\Omega$ and every $t\geq 0$},
 \end{equation}
  where $q\in \left[p,\frac{pn}{n-p}\right]$ if $p<n$, or $q\in [p,+\infty)$ if $p\geq n$.
\end{remark}
 
We conclude this section with some consequences of as\-sump\-tions (f1)--(f4) which shall be useful
 in the sequel (see, e.g., \cite{FMP} for related remarks).  First, taking into account assumption (f4), we introduce the functions
  \begin{equation}\label{a0ainfty}
   a_0 (x) := \lim_{t\to 0^{+}} \dfrac{f(x,t)}{t^{p-1}}
   \qquad\text{and}\qquad
   a_{\infty} (x) := \lim_{t\to +\infty} \dfrac{f(x,t)}{t^{p-1}},
  \end{equation}
 noticing that the first one is allowed to be identically equal to $+\infty$.

Then, we notice that: 
 \begin{enumerate}
  \item by combining (f2) and (f4), we get that
  \begin{equation}\label{Defcf}
 \frac{f(x,t)}{t^{p-1}} \geq f(x,1) \geq - \|f(\cdot,1)\|_{L^{\infty}(\Omega)} =: -c_f > -\infty,
 \end{equation} 
 for a.e.\,$x \in \Omega$ and every $t \in (0,1]$. In particular, from (f1) and \eqref{Defcf} we get
 \begin{equation}\label{Segnof}
 \text{$f(x,0)\geq 0$ for a.e.\,$x\in\Omega$}.
\end{equation}

 \item Using again assumption (f4), we have
 \begin{equation*}
  a_0(x)\geq \frac{f(x,t)}{t^{p-1}} \geq a_\infty(x) 
 \end{equation*}
 for a.e.\,$x \in \Omega$ and every $t>0$. In particular, by \eqref{Defcf} we get
 \begin{equation} \label{eq:azerobd}
  a_0(x)\geq -c_f\geq a_\infty(x)\quad\text{for a.e.\,$x\in\Omega$}.
 \end{equation}
 \end{enumerate}

\section{Strong maximum principle}\label{sec:SMP}

While dealing with \emph{nonnegative} weak solutions $u$ of  $\LL_{p,s}u = f(x,u)$ (with zero-boun\-da\-ry conditions),
 it should be desirable to know that either
\[
\text{$u\equiv 0$ in $\Omega$ \qquad or  \qquad $u > 0$ a.e.\,in $\Omega$},
\]
 namely, that a {\it strong maximum principle} holds.
 
The next theorem shows that this is indeed true in our context.
 \begin{theorem} \label{thm.SMPWeak}
Let $f$ satisfy {\rm (f1)--(f3)} and let $u\in\mathbb{X}^+_p(\Omega)$ satisfy
  identity \eqref{eq.defWeakSol} for every function $\varphi\in \mathbb{X}_p(\Omega)$.
  Then, either $u\equiv 0$ or $u > 0$ almost everywhere in $\Omega$. 
 \end{theorem}
\begin{remark}
Actually, as it will be clear from the proof, we prove a logarithmic inequality - inequality 
\eqref{eq.estimFinalSMP} - which implies, when $u$ is not the trivial function, that 
the set $\{x\in \Omega\,:\,u(x)=0\}$ has zero $W^{1,p}-$capacity, as in \cite[Theorem 2.4]{LuPra}.

\end{remark}
 \begin{proof}[Proof of Theorem $\ref{thm.SMPWeak}$]
  We suppose that there exists a set $\mathcal{Z}\subseteq\Omega$, with positive
  Lebesgue measure, such that $u\equiv 0$ a.e.\,on $\mathcal{Z}$.
  Then, we claim that
  \begin{equation} \label{eq.claimB}
   \text{$\exists\,\,x_0\in\Omega,\,R > 0$ such that $u\equiv 0$ a.e.\,on $B(x_0,R)\Subset\Omega$}.
  \end{equation}
  Taking this claim for granted for a moment, 
  we now choose a nonne\-ga\-ti\-ve fun\-ction $\varphi\in C_0^\infty(\Omega)$ satisfying
  the properties
  $$\mathrm{supp}(\varphi)\subseteq B(x_0,R)\qquad
  \text{and}\qquad\int_{B(x_0,R)}\varphi\, dx = 1.$$
  Using $\varphi$ as a test function in \eqref{eq.defWeakSol},
  from \eqref{eq.claimB} we get
  \begin{equation} \label{eq.estimconf}
  \begin{split}
   & \int_{B(x_0,R)}f(x,0)\varphi\, dx
   = \int_{\Omega}f(x,u)\varphi\, dx
   \\
   & \qquad = \iint_{\R^{2n}}
   \!\!\!\frac{|u(x)-u(y)|^{p-2}(u(x)-u(y))(\varphi(x)-\varphi(y))}{|x-y|^{n+ps}}\, dx\, dy
   \\
   &  
   \qquad = -2\int_{
   \Omega\setminus B(x_0,R)}\int_{B(x_0,R)}
   \frac{u(x)^{p-1}\varphi(y)}{|x-y|^{n+ps}}\, dx\, dy \\
   & \qquad \leq -\frac{2}{\mathrm{diam}(\Omega)^{n+ps}}
   \int_{\Omega\setminus B(x_0,R)}u(x)^{p-1}\, dx.
  \end{split}
  \end{equation}
  On the other hand, since $\varphi$ is nonnegative, by 
  \eqref{Segnof} we have
  \begin{equation} \label{eq.signintf}
   \int_{B(x_0,R)}f(x,0)\varphi\, dx\geq 0.   
  \end{equation}
  Gathering \eqref{eq.estimconf} and \eqref{eq.signintf}, we obtain
  $$\int_{\Omega\setminus B(x_0,R)}u(x)^{p-1}\, dx = 0$$
  and thus $u\equiv 0$ a.e.\,on $\Omega\setminus B(x_0,R)$ (remind that, by assumption,
  $u\in\mathbb{X}^+_p(\Omega)$). O\-wing to \eqref{eq.claimB},
  we then conclude that $u\equiv 0$ a.e.\,in $\Omega$, as desired.
  \vspace*{0.05cm}
  
  To complete the proof, we are left to 
  show the claim \eqref{eq.claimB}. 
  First of all, since we are assuming that 
  $\mathcal{Z}\subseteq\Omega$ has positive Lebesgue measure,
  it is possible to find a point $x_0\in\Omega$ and some $R > 0$ such that
  $$B(x_0,2R)\Subset \Omega\qquad\text{and}\qquad
  |\mathcal{Z}\cap B(x_0,R)| > 0.$$
  Moreover, we choose a nonnegative function $\varphi\in C_0^\infty(\Omega)$ 
  such that $\varphi\equiv 1$ a.e.\,in $B(x_0,R)$ and $\mathrm{supp}(\varphi)\subseteq 
  B(x_0,2R)$. For every fixed $\varepsilon > 0$, we then set
  $$\varphi_\varepsilon := \frac{\varphi^p}{(u+\varepsilon)^{p-1}}.$$
  Since $\varphi\in C_0^\infty(\Omega)$ and
  $u\in\mathbb{X}_p^+(\Omega)$, it is easy to recognize that
  $\varphi_\varepsilon \in \mathbb{X}_p(\Omega)$ (see, e.g., \cite[Lem.\,2.3]{MPV}); 
  we can then use $\varphi_\varepsilon$ as a test function
  in \eqref{eq.defWeakSol}, obtaining
  \begin{equation} \label{eq.toestimSMP}
  \begin{split}
   & (p-1)\int_{\Omega}\frac{|\nabla u|^{p}}{(u+\varepsilon)^{p}}\varphi^p\, dx
   \\
   & \qquad\leq \iint_{\R^{2n}}
   \frac{|u(x)-u(y)|^{p-2}(u(x)-u(y))
   (\varphi_\varepsilon(x)-\varphi_\varepsilon(y))}{|x-y|^{n+ps}}\, dx\, dy
   \\
   & \qquad\quad
   + p\int_{\Omega}\frac{|\nabla u|^{p-1}|\nabla\varphi|}{(u+\varepsilon)^{p-1}}
   \varphi^{p-1}\, dx - \int_{\Omega}f(x,u)\,\frac{\varphi^p}{(u+\varepsilon)^{p-1}}\, dx.
  \end{split}
  \end{equation}
  We now turn to provide \emph{ad-hoc} estimates
  for the three integrals in the right-hand side of \eqref{eq.toestimSMP}.
  First of all, in the proof of \cite[Lem.\,1.3]{CKP} it is showed that
 \begin{align*}
&\frac{|u(x)-u(y)|^{p-2}(u(x)-u(y)) (\varphi_\varepsilon(x)-\varphi_\varepsilon(y))}{|x-y|^{n+ps}}\\
&\qquad \leq -K\frac{1}{|x-y|^{n+ps}}\left| \log\left( \frac{u(x)+\varepsilon}{u(y)+\varepsilon}\right)\right|^p\varphi^p(y)+K\frac{|\varphi(x)-\varphi(y)|^p}{|x-y|^{n+ps}}
\end{align*}
for some positive constant $K = K_p > 0$.
Hence, by integrating we find
  \begin{equation} \label{eq.estimNonlocalSMP}
  \begin{split}
   & \iint_{\R^{2n}}\frac{|u(x)-u(y)|^{p-2}(u(x)-u(y))
   (\varphi_\varepsilon(x)-\varphi_\varepsilon(y))}{|x-y|^{n+ps}}\, dx\, dy
   \\
   & \qquad\quad \leq K\iint_{\R^{2n}}\frac{|\varphi(x)-\varphi(y)|^p}{|x-y|^{n+ps}} dx\, dy.
   \end{split}
  \end{equation}
 Moreover, using the weighted Young inequality, for every $\varepsilon > 0$ one has
  \begin{equation} \label{eq.YoungSMP}
   \begin{split}
    & p\int_{\Omega}\frac{|\nabla u|^{p-1}|\nabla\varphi|}{(u+\varepsilon)^{p-1}}
   \varphi^{p-1}\, dx \\
   & \qquad\quad
   \leq \frac{p-1}{2}\int_{\Omega}\frac{|\nabla u|^{p}}{(u+\varepsilon)^{p}}\varphi^p\, dx
  + 2^{p-1}\int_{\Omega}|\nabla \varphi|^p\, dx.
   \end{split}
  \end{equation}
  As for the remaining integral, 
  we proceed essentially as in \cite[Lem.\,2.4]{MPV}:
  by ex\-ploi\-ting \eqref{Defcf}
  and \eqref{Segnof}, we have
  the following chain of inequalities:
  \begin{equation} \label{eq.lastMPV}
   \begin{split}
    & - \int_{\Omega}f(x,u)\,\frac{\varphi^p}{(u+\varepsilon)^{p-1}}\, dx
    =
    - \int_{\Omega\cap\{u = 0\}}f(x,0)\,\frac{\varphi^p}{\varepsilon^{p-1}}\, dx
    \\
    & \qquad\qquad\quad
    - \int_{\Omega\cap\{0<u<1\}}\frac{f(x,u)\,\varphi^p}{(u+\varepsilon)^{p-1}}\, dx
    - \int_{\Omega\cap\{u\geq 1\}}\frac{f(x,u)\,\varphi^p}{(u+\varepsilon)^{p-1}}\, dx
    \\
    & \qquad \leq 
    c_f\int_{\Omega\cap\{0<u<1\}}
    \frac{u^{p-1}}{(u+\varepsilon)^{p-1}}\,\varphi^p\, dx
    +c_p\int_{\Omega\cap\{u\geq 1\}}
    \frac{1+u^{p-1}}{(u+\varepsilon)^{p-1}}\,\varphi^p\, dx \\[0.2cm]
    & \qquad
    \leq (c_f+2c_p)\|\varphi\|^p_{L^p(\Omega)}.
   \end{split}
  \end{equation}
  Gathering \eqref{eq.toestimSMP}--\eqref{eq.lastMPV}, 
  we then obtain
  \begin{equation} \label{eq.estimFinalSMP}
   \begin{split}
    & \int_{B(x_0,R)}\Big|\nabla \log\Big(1+\frac{u}{\varepsilon}\Big)\Big|^p\, dx
    = \int_{B(x_0,R)}\frac{|\nabla u|^{p}}{(u+\varepsilon)^{p}}\, dx \\
    & \qquad\qquad
    \leq 
    \int_{B(x_0,R)}\frac{|\nabla u|^{p}}{(u+\varepsilon)^{p}}\,\varphi^p\, dx
    \leq \kappa,
   \end{split}
  \end{equation}
  where $\kappa = \kappa_\varphi > 0$ is a suitable constant
  \emph{independent of $\varepsilon$}.

 With \eqref{eq.estimFinalSMP} at hand, we are finally ready to prove
  \eqref{eq.claimB}. In fact, recalling that 
  $$E := \mathcal{Z}\cap B(x_0,R)$$ 
  has positive Lebesgue measure and $u\equiv 0$ a.e.\,in $E$, 
  by 
  Chebyshev's inequality and the Poincar\'{e} inequality in \cite[Theorem 13.27]{Leoni}
  for every $t>0$ we have
  \begin{equation} \label{eq.afterPoincCheb}
  \begin{split}
   &
   \Big|\log\Big(1+\frac{t}{\varepsilon}\Big)\Big|^p\cdot
   \big|\{u\geq t\}\cap B(x_0,R)\big| 
    \leq \int_{B(x_0,R)}
    \Big|\log\Big(1+\frac{u}{\varepsilon}\Big)\Big|^p\, dx \\ 
    & \qquad
    = 
    \int_{B(x_0,R)}
    \Big|\log\Big(1+\frac{u}{\varepsilon}\Big)
    - m_E\Big|^p\, dx \\
   & \qquad \leq C_P\,
   \int_{B(x_0,R)}\Big|\nabla \log\Big(1+\frac{u}{\varepsilon}\Big)\Big|^p\, dx
   \leq \kappa'.
   \end{split}
  \end{equation}
  where $m_E$ is the mean of $v:=\log(1+u/\varepsilon)\in W^{1,p}(\R^n)$ 
  on the set $E$, that is,
  $$m_E := \frac{1}{|E|}\int_E\log\Big(1+\frac{u}{\varepsilon}\Big)\, dx = 0.$$
  As a consequence, since identity 
  \eqref{eq.afterPoincCheb} holds for every $\varepsilon > 0$ and the constant
  $\kappa'$ is \emph{independent of $\varepsilon$}, we readily infer that
  $$\big|\{u\geq t\}\cap B(x_0,R)\big| =0\qquad\text{for every $t > 0$}.$$
  This obviously implies that $u\equiv 0$ a.e.\,in $B(x_0,R)$, and the proof is complete.
 \end{proof}
 From Theorem \ref{thm.SMPWeak}, we immediately deduce the following result.
 \begin{corollary} \label{cor.posSol}
Let $f$ satisfy {\rm (f1)--(f3)} and let $u\in\mathbb{X}_p(\Omega)$ be a weak solution of \eqref{eq.pbDirSec2}.
  Then, 
  $$\text{$u > 0$ a.e.\,in $\Omega$}.$$
 \end{corollary}
 \begin{proof}
  Since $u$ is a weak solution of \eqref{eq.pbDirSec2}, it follows
  from Definition \ref{def.WeakSol} that
  \begin{itemize}
   \item[(a)] $u\in\mathbb{X}^+_p(\Omega)$ (i.e., $u\geq 0$ a.e.\,in $\Omega$);
   \vspace*{0.05cm}
   
   \item[(b)] $|\{x\in\Omega:\,u(x) > 0\}| > 0$. 
  \end{itemize}
  In particular, from (b) we get that $u$ is not identically vanishing (a.e.) in $\Omega$,
  and the conclusion
  follows immediately from Theorem \ref{thm.SMPWeak}.
 \end{proof}

 \begin{remark} \label{rem.Eigenvalue}
 By carefully scrutinizing the proof of Theorem \ref{thm.SMPWeak}, one can easily
 check that the properties of $f$ which have \emph{actually} played a role 
 are:
 \begin{enumerate}
  \item $f(x,0)\geq 0$ for a.e.\,$x\in\Omega$;
  \item $f(x,t)\geq -c_ft^{p-1}$ for a.e.\,$x\in\Omega$ and every $0<t<1$;
  \item $f(x,t) {\geq} c_p(1+t^{p-1})$ for a.e.\,$x\in\Omega$ and every $t\geq 1$.
 \end{enumerate}
 As a consequence, the strong maximum principle in Theorem \ref{thm.SMPWeak} holds
 for weak solutions of every boundary-value problem of the  type
 \begin{equation} \label{eq.generalDirPb}
 \begin{cases}
  \LL_{p,s}u = g(x,u) & \text{in $\Omega$}, \\
      u \gneqq 0 & \text{in $\Omega$}, \\
  u\equiv 0 & \text{in $\R^n\setminus\Omega$}.
  \end{cases}
  \end{equation}
 where $g:\Omega\times\R\to\R$ is a Carath\'{e}odory function satisfying
 (1)--(3) and the growth condition in \eqref{cresc}.

 A remarkable example of a map $g$ satisfying conditions (1)--(3) above is
 \begin{equation}\label{glambda}
g_{\lambda}(x,t) := \big(-a(x)+\lambda)|t|^{p-2}t,
\end{equation}
 where $\lambda\in\R$ and $a\in L^\infty(\Omega)$. 
 The boundary-value problem associated with
 this function $g_\lambda$
 is the (Di\-ri\-chlet) $\LL_{p,s}$-eigenvalue problem 
\[
 \begin{cases}
  \LL_{p,s}u + a(x)|u|^{p-2}u = \lambda |u|^{p-2}u & \text{in $\Omega$}, \\
  u\equiv 0 & \text{in $\R^n\setminus\Omega$},
  \end{cases}
\]
 which shall be extensively studied in Section \ref{sec.Eigenvalue}.
 \end{remark}

 \section{Uniqueness and boundedness of weak solutions} \label{sec.uniqueness}
 The aim of this section is to establish \emph{uniqueness and boundedness}
 of weak solutions to problem \eqref{eq.pbDirSec2}.

We start by proving that weak solutions are globally bounded. We stress this ``re\-gu\-la\-ri\-ty result"
 requires $f$ to satisfy \emph{only} assumptions (f1)-- (f3).
 \begin{theorem} \label{thm:uzeroglobalBd}
  Let $u_0\in\mathbb{X}_p(\Omega)$ be a nonnegative weak solution of 
  \eqref{eq.pbDirSec2} with $f$ satisfying \emph{(f1)}--\emph{(f3)}.
Then $u_0\in L^\infty(\Omega)$.
\end{theorem}   
\begin{proof}
To begin with, we arbitrarily fix $\delta \in (0,1)$ and we set
   $$\tilde{u}_0 := \delta^{1/(p-1)}u_0.$$
Then $\tilde u_0$ solves
   \begin{equation} \label{eq:uzeroPDE}
   \begin{cases} 
 -\Delta_p \tilde u_0+ (-\Delta)^s_p\tilde u_0 = \delta f(x,u_0) & \text{in $\Omega$}, \\
   \tilde u_0 \gneqq 0 & \text{in $\Omega$}, \\
   \tilde u_0 \equiv 0 & \text{in $\R^n\setminus \Omega$}.
   \end{cases}
   \end{equation}
  Now, for every $k\geq 0$, we define $C_k := 1-2^{-k}$ and
$$
v_k:=\tilde{u}_0-C_k, \quad w_k:=(v_k)_+:=\max\{v_k,0\},\quad U_k:= \|w_k\|_{L^{p}(\Omega)}^p.
$$
We explicitly point out that, in view of these definitions, one has
\begin{itemize}
 \item[(a)] $\|\tilde{u}_0\|^p_{L^p(\Omega)} = \delta^{p'}\,\|u_0\|^p_{L^p(\Omega)}$
 (with $1/p' = 1-1/p$);
 \vspace*{0.05cm}
 
 \item[(b)] $w_0 = v_0 = \tilde{u}_0$ (since $C_0 = 0$);
 \vspace*{0.05cm}
 
 \item[(c)] $v_{k+1}\leq v_{k}$ and $w_{k+1}\leq w_{k}$ (since $C_k < C_{k+1}$).
  \end{itemize}
  We now observe that, since $u_0\in\mathbb{X}_p(\Omega)\subseteq W^{1,p}(\R^n)$, we
  have $v_k\in W^{1,p}_{\mathrm{loc}}(\R^n)$; fur\-ther\-mo\-re, since $\tilde{u}_0 =  u_0\equiv 0$
  a.e.\,in $\R^n\setminus\Omega$, one also has
  $$v_k = \tilde{u}_0-C_k = -C_k < 0\quad\text{on $\R^n\setminus\Omega$},$$
  and thus $w_k = (v_k)_+\in\mathbb{X}_p(\Omega)$
  (remind that $\Omega$ is bounded).
  We are then entitled to use the function $w_k$ as a \emph{test function}
  in \eqref{eq:uzeroPDE}, obtaining
  \begin{equation} \label{eq:PDEsolveduzerotilde}
  \begin{split}
   \int_{\Omega}|\nabla \tilde u_0|^{p-2}\langle \nabla \tilde{u}_0,\nabla w_k\rangle\, dx
   & +  \iint_{\R^{2n}}\frac{J_p(\tilde{u}_0(x)-\tilde{u}_0(y))(w_k(x)-w_k(y))}{|x-y|^{n+2s}}\, dx\, dy
   \\
   & = \delta \int_{\Omega}f(x,u_0)w_k\, dx.
   \end{split}
  \end{equation}
To proceed further, we notice that for any measurable function $z$ and for 
(almost every) couple of points $x,y\in\R^n$, one has
\[
|z_+(x)-z_+(y)|^p\leq |z(x)-z(y)|^{p-2}(z(x)-z(y))(z_+(x)-z_+(y))
\]
see \cite[Equation (14)]{frapa} or \cite[Equation (16)]{mpl}, so that, by choosing $z=v_k$,  since 
$$v_k(x)-v_k(y)=\tilde u_0(x)-\tilde u_0(y),$$ 
we find
\begin{equation}\label{EAFS2}
|w_{k}(x)-w_{k}(y)|^p \leq |\tilde u_0(x)-\tilde u_0(y)|^{p-2}(\tilde u_0(x)-\tilde u_0(y))(w_k (x)-w_k(y)).
\end{equation}
  Moreover, taking into account the definition of $w_k$, we get
  \begin{equation} \label{eq:localpart}
  \begin{split}
   \int_\Omega |\nabla \tilde u_0|^{p-2}\langle \nabla\tilde{u}_0,\nabla w_k\rangle \, dx 
   & =
	\int_{\Omega\cap\{ \tilde{u}_0 > C_k\}}|\nabla \tilde u_0|^{p-2}
	\langle \nabla\tilde{u}_0,\nabla v_k\rangle\, dx  \\
	& =
	\int_\Omega|\nabla w_{k}(x)|^p\, dx.
	\end{split}
  \end{equation}
  Gathering \eqref{eq:PDEsolveduzerotilde}-\eqref{eq:localpart}
  and assumption (f3), we obtain
  \begin{equation} \label{eq:estimnablawkI} 
  \begin{split}
    \int_\Omega|\nabla w_{k}|^p\, dx
    & \leq \delta \int_{\Omega}|f(x,u_0)|\,w_k\, dx \\
    & \leq c \int_{\Omega}(\delta+\delta u_0^{p-1})\,w_k\, dx 
   \leq c\int_{\Omega}(1+\tilde{u}_0^{p-1})w_k\, dx,
   \end{split}
  \end{equation}
since $\delta<1$.
We then recall that, for every $k\geq 1$, one has
\begin{equation} \label{eq:tildeuzeroleqwk}
\tilde u_0(x)<(2^k-1)w_{k-1}(x) \quad \text{for} \quad x\in \lbrace w_{k}>0 \rbrace,
\end{equation}
and the inclusions 
\begin{equation}\label{HHA} 
\{ w_{k}>0\} = \{\tilde{u}_0 > C_k\} 	\subseteq \{w_{k-1}>{2^{-k}}\}
\end{equation}
hold true for every $k\geq 1$, see \cite{frapa} or \cite{mpl}.
By combining \eqref{eq:tildeuzeroleqwk} and \eqref{HHA}  with \eqref{eq:estimnablawkI}, and taking into account that $w_k\leq w_{k-1}$ a.e.\,in $\R^n$, for every $k\geq 1$, we get
\begin{equation} \label{eq:estimnablawkII}
   \begin{aligned}
 \int_\Omega|\nabla w_{k}|^p\, dx     &
    \leq c\,\int_{\{w_k > 0\}}(1+\tilde{u}_0^{p-1}) w_k\, dx \\[0.1cm]
    &  \leq c\,\,\int_{\{w_k > 0\}}
\left[w_{k-1}+(2^k-1)^{p-1}w_{k-1}^{p}\right] dx \\[0.1cm]
    &    \leq c\int_{\{w_{k-1} > 2^{-k}\}}\left[2^{k(p-1)}w_{k-1}^p+(2^k-1)^{p-1}w_{k-1}^p\right]\, dx\\
   & \leq c\,2^{kp}\,\int_{\{w_{k-1} > 2^{-k}\}}w_{k-1}^p dx \\
   & \leq c\,2^{kp}\int_{\Omega} w_{k-1}^p dx
   = c\,2^{kp}U_{k-1}.
\end{aligned}
\end{equation}
 We now estimate from below the term $U_{k-1}$ in the right-hand side
 of \eqref{eq:estimnablawkII}. To this end we first observe that,
	as a consequence of \eqref{HHA}, we obtain
	\begin{equation} \label{eq:Uklower}
	 \begin{split}
	 U_{k-1} & = \int_{\Omega}w_{k-1}^p\, dx
	 \geq \int_{\{w_{k-1} > 2^{-k}\}}w_{k-1}^p\, dx \\[0.15cm]
	 & \geq 2^{-kp}\,\big|\{w_{k-1} > 2^{-k}\}\big|
	 \geq 2^{-kp}\big|\{w_k > 0\}|.
	 \end{split}
	\end{equation}
	Using the H\"older inequality (with exponents $p^*/p$ and $n/p$), jointly with the
	Sobolev inequality, from 
	\eqref{eq:estimnablawkII}-\eqref{eq:Uklower}
	we obtain the following estimate:
	\begin{equation} \label{eq:finaleperconcludere}
	 \begin{split}
	 U_k &= \|w_k\|^p_{L^p(\Omega)}
	 \leq \bigg(\int_\Omega w_k^{p^*}\, dx\bigg)^{p/{p^*}}\,
	 \big|\{w_k > 0\}\big|^{p/n} \\[0.1cm]
	 & \leq \mathbf{c}_S\,\int_\Omega|\nabla w_k|^p\, dx
	 \cdot
	 \big|\{w_k > 0\}\big|^{p/n} \\[0.1cm]
	 & \leq \mathbf{c}_S\,\big(c\,2^{kp}\,U_{k-1}\big)\,\big(
	 2^{kp}U_{k-1}\big)^{p/n} \\[0.2cm]
	 & = \mathbf{c}'\,\big(2^{p+p^2/n}\big)^{k-1}\,U_{k-1}^{1+p/n}
	 \qquad (\text{with $\mathbf{c}' := c\,2^{p+p^2/n}\,\mathbf{c}_S$}),
	 \end{split}
	\end{equation}
for every $k\geq 1$, where $\mathbf{c}_S$ is given by the Sobolev inequality.

Estimate \eqref{eq:finaleperconcludere} can be re-written as
\[
U_k\leq \mathbf{c}'\eta^{k-1}U_{k-1}^{1+p/n},
\]
where
\[
\eta := 2^{p+p^2/n} > 1.
\]
Hence, from \cite[Lem.\,7.1]{Giusti} we get that $U_k\to 0 $ as $k\to\infty$, provided that
\[
U_0 = \|\tilde{u}_0\|^p_{L^p(\Omega)} = \delta^{p'}\|u_0\|^p_{L^p(\Omega)}
	< (\mathbf{c}')^{-n/p}\,\eta^{-n^2/p^2}. 
\]
As a consequence, if $\delta > 0$ is small enough, we obtain
	$$0 = \lim_{k\to\infty} U_k = 
	\lim_{k\to\infty}\int_\Omega(\tilde{u}_0-C_k)_+^2\, dx
	= \int_{\Omega}(\tilde{u}_0-1)_+^2\, dx.$$
	Bearing in mind that $\tilde{u}_0 =\delta^{1/(p-1)}u_0$ (and $u_0\geq 0$), we then get
	$$0\leq u_0\leq \frac{1}{\delta^{1/(p-1)}}\qquad\text{a.e.\,in $\Omega$},$$
	from which we conclude that $u_0\in L^\infty(\Omega)$.
 \end{proof} 

\begin{remark}
We notice that an analogous result still holds true, with suitable adaptations in the powers of $u_0$ or $w_k$ in the right-hand sides of the inequalities in the above proof, also when $f$ satisfies (f1), (f2) and \eqref{cresc}. However, in view of the $p-$linear growth in the Brezis-Oswald theorem, we preferred to maintain such a case for the presentation of the proof. 
\end{remark}

\medskip

We are now ready to state and prove the main result of this section.
 \begin{theorem} \label{thm.uniqueness}
Let $f$ satisfy {\rm (f1)--(f5)}. Then there exists \emph{at most} one weak so\-lu\-tion
  $u\in\mathbb{X}_p(\Omega)$ of problem \eqref{eq.pbDirSec2}.
 \end{theorem}
 In order to prove Theorem \ref{thm.uniqueness}, we need
 the following elementary lemma.
 \begin{lemma} \label{lem.elemntare}
  Let $v,w\in\R^n$  and set
  $$\mathcal{A}_p(v,w) := |v|^p+(p-1)|w|^p-p|w|^{p-2}\langle v,w\rangle.$$
  Then, $\mathcal{A}_p(v,w)\geq 0$.
 \end{lemma}
 \begin{proof}
  We first notice that, if $v = 0$ or $w = 0$, the conclusion
  of the lemma is trivial. We then assume that
  $v,w\neq 0$, and we let $t > 0$ be such that
  \begin{equation} \label{eq.choicetLemma}
   |w| = t|v|.
   \end{equation}
 Using Cauchy-Schwarz's inequality and \eqref{eq.choicetLemma}, we have 
  \begin{align*}
   \mathcal{A}_p(v,w) & \geq |v|^p+(p-1)|w|^p-p|w|^{p-1}|v|
   \\
   & = |v|^p\big(1+(p-1)t^p-p\,t^{p-1}\big) =: |v|^p\cdot\ell_p(t).
   \end{align*}
  From this, since an elementary computation shows that
  $$\ell_p(s)\geq \ell_p(1) = 0\qquad\text{for every $s\geq 0$},$$
  we readily conclude that $\mathcal{A}_p(v,w)\geq 0$, as desired.
 \end{proof}
 Thanks to Lemma \ref{lem.elemntare}, we can proceed with the proof of Theorem \ref{thm.uniqueness}.
 
 \begin{proof} [Proof of Theorem $\ref{thm.uniqueness}$]
  Let $u_1,u_2\in\mathbb{X}_p(\Omega)$ be two solutions of 
  \eqref{eq.pbDirSec2}. 
  In order to show that $u_1 = u_2$ a.e.\,in $\Omega$, we arbitrarily
  fix $\varepsilon > 0$ and we define
  $$\varphi_{1,\varepsilon} :=
  r_{1,\varepsilon}-u_1,\qquad
  \varphi_{2,\varepsilon}
  := r_{2,\varepsilon}-u_2,
  $$
  where
  $$r_{1,\varepsilon}
  := \frac{u_2^p}{(u_1+\varepsilon)^{p-1}},\qquad
  r_{2,\varepsilon} := \frac{u_1^p}{(u_2+\varepsilon)^{p-1}}.$$
  Taking into account that $u_1,u_2\in\mathbb{X}_p(\Omega)$, $u_1,u_2\geq 0$ a.e.\,in $\Omega$
  and that $u_1, u_2$ are \emph{globally bounded in $\Omega$} 
  (as it follows Theorem \ref{thm:uzeroglobalBd}),
   we readily infer that
  $\varphi_{i,\varepsilon}\in\mathbb{X}_p(\Omega)$ for every
  $\varepsilon > 0$ and $i = 1,2$. Hence,  using $\varphi_{i,\varepsilon}$
  as a test function in \eqref{eq.defWeakSol} for $u_i$ and adding
  the resulting identities, we get
  \begin{equation} \label{eq.mainestimtest}
   \begin{split}
   & \int_{\Omega}
   |\nabla u_1|^{p-2}\langle \nabla u_1,\nabla\varphi_{1,\varepsilon}\rangle\, dx
   + \int_{\Omega}|\nabla u_2|^{p-2}
   \langle \nabla u_2,\nabla\varphi_{2,\varepsilon}\rangle\, dx
   \\
   & \qquad
   + \iint_{\R^{2n}}
   \frac{|u_1(x)-u_1(y)|^{p-2}(u_1(x)-u_1(y))(\varphi_{1,\varepsilon}(x)
   -\varphi_{1,\varepsilon}(y))}{|x-y|^{n+ps}}\, dx\, dy
   \\
   & \qquad
   + \iint_{\R^{2n}}
   \frac{|u_2(x)-u_2(y)|^{p-2}(u_2(x)-u_2(y))(\varphi_{2,\varepsilon}(x)
   -\varphi_{2,\varepsilon}(y))}{|x-y|^{n+ps}}\, dx\, dy \\
   & = \int_{\Omega}\big(f(x,u_1)\varphi_{1,\varepsilon}+
   f(x,u_2)\varphi_{2,\varepsilon}\big)\, dx.
   \end{split}
  \end{equation}
  Now, a direct computation based on the very definition of
  $\varphi_{i,\varepsilon}$ gives 
  \begin{align*}
  & \int_{\Omega}
   |\nabla u_1|^{p-2}\langle \nabla u_1,\nabla\varphi_{1,\varepsilon}\rangle\, dx
   + \int_{\Omega}|\nabla u_2|^{p-2}
   \langle \nabla u_2,\nabla\varphi_{2,\varepsilon}\rangle\, dx \\
   & \qquad
     = - \int_\Omega
   \!\!\!\mathcal{A}_p\Big(\nabla u_1,\frac{u_1}{u_2+\varepsilon}\nabla u_2\Big)\, dx
      -\int_\Omega
   \mathcal{A}_p\Big(\nabla u_2, \frac{u_2}{u_1+\varepsilon}\nabla u_1\Big)\, dx,
  \end{align*}
   where $\mathcal{A}_p$ is as in Lemma \ref{lem.elemntare}; as a consequence,
   taking into account that 
   $\mathcal{A}_p(\cdot,\cdot)\geq 0$
   (as we know from Lemma \ref{lem.elemntare}), identity
   \eqref{eq.mainestimtest} boils down to
   \begin{equation} \label{eq.topasslimitUniq}
   \begin{split} 
    & \int_{\Omega}\big(f(x,u_1)\varphi_{1,\varepsilon}+
   f(x,u_2)\varphi_{2,\varepsilon}\big)\, dx \\
   & \qquad
   \leq \iint_{\R^{2n}}
   \frac{J_p(u_1(x)-u_1(y))(r_{1,\varepsilon}(x)-r_{1,\varepsilon}(y))}
   {|x-y|^{n+ps}}\, dx\, dy
   \\
   &
   \qquad
   + \iint_{\R^{2n}}
   \frac{J_p(u_2(x)-u_2(y))(r_{2,\varepsilon}(x)
   -r_{2,\varepsilon}(y))}{|x-y|^{n+ps}}\, dx\, dy \\
   & 
   \qquad
   - \iint_{\R^{2n}}
   \frac{J_p(u_1(x)-u_1(y))
   (u_1(x)-u_1(y))}{|x-y|^{n+ps}}\, dx\, dy
   \\
   & \qquad
   -\iint_{\R^{2n}}
   \frac{J_p(u_2(x)-u_2(y))
   (u_2(x)-u_2(y))}{|x-y|^{n+ps}}\, dx\, dy \\[0.2cm]
   & \qquad
   =: \mathrm{I}_{1,\varepsilon}+\mathrm{I}_{2,\varepsilon}-\mathrm{J}_1
   -\mathrm{J}_2,
      \end{split}
   \end{equation}
   where we have introduced the standard notation
   $$J_p(t) := |t|^{p-2}t\qquad(t\in\R).$$
   We now aim at passing to the limit as $\varepsilon\to 0^+$ in the above
   \eqref{eq.topasslimitUniq}.
   \vspace*{0.05cm}
   
   To this end, we first remind the following 
   {discrete Picone inequality}:
   \emph{for every fixed $p\in (1,+\infty)$ and every $a,b,c,d\in[0,+\infty)$ with $a,b>0$, one has
  $$
   J_p(a-b)\left(\frac{c^p}{a^{p-1}}-\frac{d^p}{b^{p-1}}\right)\leq |c-d|^p,
  $$
  and equality holds if and only if}
   $$ad=bc$$
   (for a proof of this inequality see, e.g., \cite[Proposition 4.2]{BF} or  
   \cite[Proposition 2.2]{BS}).
   By using this inequality, we have
   \begin{itemize}
    \item[(i)]  $J_p(u_1(x)-u_1(y))
    (r_{1,\varepsilon}(x)-r_{1,\varepsilon}(y))
    \leq |u_2(x)-u_2(y)|^p$;
    \vspace*{0.08cm}
    
    \item[(ii)] $J_p((u_2(x)-u_2(y))
    (r_{2,\varepsilon}(x)-r_{2,\varepsilon}(y))
    \leq |u_1(x)-u_1(y)|^p$.
   \end{itemize}
   Hence, we are entitled to apply the Fatou lemma for the integrals 
   $\mathrm{I}_{1,\varepsilon},
   \mathrm{I}_{2,\varepsilon}$, obtaining
   \begin{equation} \label{eq.limitrhsUniq}
   \begin{split}
    & \limsup_{\varepsilon\to 0^+}
    \big(\mathrm{I}_{1,\varepsilon}+\mathrm{I}_{2,\varepsilon}-\mathrm{J}_1
   -\mathrm{J}_2\big) \\
   & \quad \leq
   \iint_{\R^{2n}}
   \frac{J_p(u_1(x)-u_1(y))}{|x-y|^{n+ps}}\bigg(
   \frac{u_2^p}{u_1^{p-1}}(x)-\frac{u_2^p}{u_1^{p-1}}(y)\bigg)\, dx\, dy \\
   & \quad
   +\iint_{\R^{2n}}
   \frac{J_p(u_2(x)-u_2(y))}{|x-y|^{n+ps}}\bigg(
   \frac{u_1^p}{u_2^{p-1}}(x)-\frac{u_1^p}{u_2^{p-1}}(y)\bigg)\, dx\, dy \\
   & \quad
   - \iint_{\R^{2n}}\frac{|u_1(x)-u_1(y)|^{p}}{|x-y|^{n+ps}}\, dx\, dy
   -\iint_{\R^{2n}}\frac{|u_2(x)-u_2(y)|^{p}}{|x-y|^{n+ps}}\, dx\, dy
   \\[0.2cm]
   & \quad =: \kappa(u_1,u_2,p),
   \end{split}
   \end{equation}
   where $\kappa(u_1,u_2,p) \in [-\infty,0]$ again by the discrete Picone inequality
   (here, to give a meaning to the integrals when $x$ or $y$ are not 
   in $\Omega$, we have tacitly set $0/0 = 0$).
   \medskip
   
   We now turn our attention to the left hand side of \eqref{eq.topasslimitUniq}.
   Taking into account the 
   very de\-fi\-ni\-tion of $\varphi_{i,\varepsilon}$, we first 
   write
   \begin{align*}
    & \int_{\Omega}\big(f(x,u_1)\varphi_{1,\varepsilon}+
   f(x,u_2)\varphi_{2,\varepsilon}\big)\, dx
   = \int_{\Omega}f(x,u_1)\,r_{1,\varepsilon}\, dx
   +\int_{\Omega}f(x,u_2)\,r_{2,\varepsilon}\, dx
    \\
    & \qquad\qquad\qquad
    - \int_{\Omega}f(x,u_1)u_1\, dx
    -\int_{\Omega}f(x,u_2)u_2\, dx\\[0.2cm]
    & \qquad\qquad\qquad
    =: \mathrm{A}_{1,\varepsilon}+ \mathrm{A}_{2,\varepsilon}
    - \mathrm{B}_{1}- \mathrm{B}_{2}.
   \end{align*}
   Moreover, recalling the value $\rho_f > 0$ in \eqref{eq.ineqrhof}, we further split
   $\mathrm{A}_{i,\varepsilon}$ as
   $$\mathrm{A}_{i,\varepsilon} = 
   \int_{\{u_i < \rho_f\}}f(x,u_i)\,r_{{i,\varepsilon}}\, dx
   + \int_{\{u_i \geq \rho_f\}}f(x,u_i)\,r_{{i,\varepsilon}}\, dx
   =: \mathrm{A}'_{i,\varepsilon}+\mathrm{A}''_{i,\varepsilon}.$$
   Now, by assumption (f3), for every $\varepsilon > 0$ we have
\[
|f(x,u_1)\,r_{{1,\varepsilon}}|\cdot\mathbf{1}_{\{u_1\geq\rho_f\}} 
    \leq c_p\big(1+\rho_f^{1-p}\big)\,u_2^p
    \equiv c_{p,f}\,u_2^p
 \]
and, analogously, 
\[
|f(x,u_2)\,r_{{2,\varepsilon}}|\cdot\mathbf{1}_{\{u_2\geq\rho_f\}} 
    \leq c_{p,f}\,u_1^p.
\]
Thus, we can then apply the Dominated Convergence theorem, obtaining
   \begin{equation} \label{eq.limitAsecond}
   \begin{split}
   & \mathrm{A}''_1:=\lim_{\varepsilon\to 0^+}
   \mathrm{A}''_{1,\varepsilon} = \int_{\{u_1\geq\rho_f\}}
   \frac{f(x,u_1)}{u_1^{p-1}}\,u_2^p\, dx\in\R\qquad
   \text{and}\\
   & \mathrm{A}''_2:=\lim_{\varepsilon\to 0^+}
   \mathrm{A}''_{2,\varepsilon} = \int_{\{u_2\geq\rho_f\}}
   \frac{f(x,u_2)}{u_2^{p-1}}\,u_1^p\, dx\in\R.
   \end{split}
   \end{equation}
   Hence, it remains to study the behavior of $A'_{i,\varepsilon}$ when
   $\varepsilon\to 0^+$.
   
   First of all, using \eqref{eq.ineqrhof} and the fact that $r_{i,\varepsilon}$ is
   nonnegative and monotone
   in\-crea\-sing with respect to $\varepsilon$, we can apply the Beppo Levi theorem,
   obtaining
   \begin{equation} \label{eq.limiAfirstStepI}
   \begin{split}
   & \mathrm{A}'_1 := \lim_{\varepsilon\to 0^+}
   \mathrm{A}'_{1,\varepsilon} = \int_{\{u_1 < \rho_f\}}
   \frac{f(x,u_1)}{u_1^{p-1}}\,u_2^p\, dx\in [0,+\infty]\qquad
   \text{and}\\
   & \mathrm{A}'_2:=\lim_{\varepsilon\to 0^+}
   \mathrm{A}'_{2,\varepsilon} = \int_{\{u_2 < \rho_f\}}
   \frac{f(x,u_2)}{u_2^{p-1}}\,u_1^p\, dx\in [0,+\infty].
   \end{split}
   \end{equation}
   
On the other hand, going back to estimate \eqref{eq.topasslimitUniq} and 
   taking into account 
   the very definitions of 
   the integrals $\mathrm{A}_{i,\varepsilon}',\mathrm{A}_{i,\varepsilon}'',
   \mathrm{B}_{i}$, we get
   \begin{align*}
    0\leq  \mathrm{A}'_{1,\varepsilon},\mathrm{A}'_{2,\varepsilon}
    & \leq  \mathrm{A}'_{1,\varepsilon}+\mathrm{A}'_{2,\varepsilon} \\
    & \leq 
    \mathrm{I}_{1,\varepsilon}+\mathrm{I}_{2,\varepsilon}-\mathrm{J}_{1}
   -\mathrm{J}_{2}
   + \mathrm{B}_{1}+
   \mathrm{B}_{2}-\mathrm{A}''_{1,\varepsilon}-\mathrm{A}''_{2,\varepsilon}.
   \end{align*}
   Then, by letting $\varepsilon\to 0^+$ 
   with the aid of \eqref{eq.limitrhsUniq}--\eqref{eq.limitAsecond}, we obtain
   \begin{equation*} 
   \begin{split}
    0\leq \mathrm{A}'_1,\mathrm{A}'_2 \leq \mathrm{A}'_1+\mathrm{A}'_2  
   \leq \kappa(u_1,u_2,p)
   +\mathrm{B}_1+\mathrm{B}_2-\mathrm{A}_1''-\mathrm{A}_2'',
   \end{split}
   \end{equation*}
   from which we derive at once that
   \begin{equation} \label{eq.limitAprimeFinite}
    \kappa(u_1,u_2,p) > -\infty\qquad\text{and}\qquad
    \mathrm{A}'_1,\mathrm{A}'_2 <+\infty.
   \end{equation}
   Gathering 
   \eqref{eq.limitAsecond}--\eqref{eq.limiAfirstStepI}, and taking
   into account \eqref{eq.limitAprimeFinite}, we finally have
   \begin{equation} \label{eq.limitLHSUniq}
    \begin{split} 
     & \lim_{\varepsilon\to 0^+}
     \bigg(\int_{\Omega}\big(f(x,u_1)\varphi_{1,\varepsilon}+
   f(x,u_2)\varphi_{2,\varepsilon}\big)\, dx\bigg) \\
   & \qquad
   = \lim_{\varepsilon\to 0^+}
   \big(\mathrm{A}'_{1,\varepsilon}+\mathrm{A}'_{2,\varepsilon}
   +\mathrm{A}''_{1,\varepsilon}+\mathrm{A}''_{2,\varepsilon}
   -\mathrm{B}_{1}-\mathrm{B}_{2}\big) \\
   & \qquad
   = \int_{\Omega}\Big(\frac{f(x,u_1)}{u_1^{p-1}}u_2^p
   +\frac{f(x,u_2)}{u_2^{p-1}}u_1^p
   - f(x,u_1)u_1-f(x,u_2)u_2\Big)\, dx \\
   & 
   \qquad 
   = -\int_{\Omega}\Big(\frac{f(x,u_1)}{u_1^{p-1}}-\frac{f(x,u_2)}{u_2^{p-1}}
   \Big)(u_1^p-u_2^p)\, dx.
    \end{split}
   \end{equation}
  With \eqref{eq.limitrhsUniq} and \eqref{eq.limitLHSUniq} at hand, we can easily
  conclude the proof of the theorem. In\-deed, using these cited identities
  we can let $\varepsilon\to 0^+$ in \eqref{eq.topasslimitUniq}, obtaining
  \begin{align*}
   -\int_{\Omega}\Big(\frac{f(x,u_1)}{u_1^{p-1}}-\frac{f(x,u_2)}{u_2^{p-1}}
   \Big)(u_1^p-u_2^p)\, dx
   \leq \kappa(u_1,u_2,p)\leq 0.
  \end{align*}
  From this, by crucially exploiting
 assumption
  (f4) we conclude that 
  $$\text{$u_1\equiv u_2$ a.e.\,in $\Omega$},$$
  and the proof is complete.
 \end{proof}
 
\section{The eigenvalue problem} \label{sec.Eigenvalue}
 
 As announced, here we consider the eigenvalue problem associated to $ \LL_{p,s}$ in presence of a weight $a\in L^\infty(\Omega)$, namely
\begin{equation}\label{eigen}
\begin{cases}
-\Delta_pu+(-\Delta)^s_p u  +a(x)|u|^{p-2}u= \lambda|u|^{p-2}u & \mbox{in}\  \Omega, \\
u\not\equiv 0, & \mbox{in}\ \Omega,\\
u=0& \mbox{in}\ \mathbb{R}^{n}\setminus \Omega.
\end{cases}
\end{equation}

\begin{proposition}\label{Proposition3}
 Let $a  \in L^\infty(\Omega)$.
 Then, problem \eqref{eigen} admits a smallest ei\-gen\-va\-lue 
 $\lambda_1 (\mathcal{L}_{p,s}+a) \in \R$ which is simple, 
 and whose associated eigenfunctions do not change sign in $\R^n$.
 Moreover, every eigenfunction associated to an eigenvalue 
 $$\lambda > \lambda_1(\mathcal{L}_{p,s}+a)$$ 
 is nodal, i.e., sign changing.
\end{proposition}
\begin{proof}
 Let $\gamma:\xp \to \R$ be the $C^1-$functional defined as
 $$
 \gamma(u)= 
 \int_\Omega 
   |\nabla u|^p dx+ \iint_{\R^{2n}}
   \!\!\!\frac{|u(x)-u(y)|^{p}}{|x-y|^{n+ps}}\, dx\, dy+\int_\Omega a(x)  |u|^p \, dx 
 $$
 for all $u \in \xp$, and let it be constrained on the $C^1 -$ Banach manifold
 $$
  M:= \left\{ u \in \xp:\int_\Omega |u|^p \, dx=1\right\}.
 $$
Define
\begin{equation}\label{6}
 \lambda_1(\mathcal{L}_{p,s}+a) := \inf\big\{\gamma(u): u \in M\big\}.
\end{equation}
 Let $\{u_n\}_{n \ge1}\subseteq M$ be a minimizing sequence for \eqref{6}.  Since 
 $$ \lambda_1 (\mathcal{L}_{p,s}+a) \ge - \|a(x) \|_{L^\infty(\Omega)},$$ 
 we immediately get that $\{u_n\}_{n \ge 1} \subseteq \xp$ 
 is bounded and so we may assume that there exists $e_1\in M$ such that
\begin{equation}\label{7}
 u_n \rightharpoonup e_1 \; \text{ in } \xp\qquad\text{as $n\to+\infty$}.
\end{equation}
 In particular, by the Rellich-Kondrachev embedding theorem, we know that
\begin{equation}\label{conv}
u_n\to e_1 \text{ in }L^p(\Omega).
\end{equation}
By \eqref{7} and \eqref{conv}, we have
\begin{align*}
\gamma(e_1)& = 
\int_\Omega |\nabla e_1|^p dx + 
\iint_{\R^{2n}}
   \!\!\!\frac{|e_1(x)-e_1(y)|^{p}}{|x-y|^{n+ps}}\, dx\, dy +\int_\Omega a(x) |e_1|^p\, dx 
 + \\[0.2cm]
& \le 
 \liminf_{n \rightarrow +\infty}\gamma(u_n)=\lambda_1 (\mathcal{L}_{p,s}+a).
\end{align*}
Since $e_1\in M$, due to \eqref{conv}, by \eqref{6} we get 
$$
\gamma (e_1) =  \lambda_1(\mathcal{L}_{p,s}+a).
$$
 By the Lagrange multiplier rule, we infer that 
 $\lambda_1(\mathcal{L}_{p,s}+a)$ is the smallest  eigenvalue for problem \eqref{eigen}, with associated  
 eigenfunction $e_1 \in \xp$. Finally, notice that
 $$
 \gamma (|u|) \leq  \gamma (u)\; \text{ for all } u \in X^s_\beta,
 $$
 and so we may assume that $e_1 \ge0$ in $\R^n$. Since $\|e_1\|_{L^{p}(\Omega)}=1$ by construction,
 we can then apply Theorem \ref{thm.SMPWeak} and Remark \ref{rem.Eigenvalue} to conclude that 
 $$e_1 (x)> 0, \quad \textrm{for a.e. } x \in \R^n.$$
 Now, we prove that $e_1$ is simple. To this end, let
 $u \in \xp$ be another
 eigenfunction associated to $ \lambda_1(\mathcal{L}_{p,s}+a)$. We first
 claim that
 $u$ has \emph{constant sign}: in fact, taking into account that
 the eigenfunctions associated to $\lambda_1(\mathcal{L}_{p,s}+a)$
 are precisely the constrained minimizers of $\gamma$, we have
 $$\gamma(u) = \lambda_1(\mathcal{L}_{p,s}+a)\|u\|^p_{L^p(\Omega)};$$
 on the other hand, if both $\{u > 0\}$ and $\{u < 0\}$ have positive
 Lebesgue measure, by arguing exactly as in the proof of \cite[Proposition 9]{SeVa},
 we have
 $$\gamma(|u|) < \gamma(u) = \lambda_1(\mathcal{L}_{p,s}+a)\|u\|^p_{L^p(\Omega)},$$
 which is clearly in contradiction with the fact that $\lambda_1(\mathcal{L}_{p,s}+a)$
 is the minimum of $\gamma$. Hence, $u$ has constant sign in $\Omega$ and we can
 assume that $u\geq 0$ a.e.\,in $\Omega$; from this, using once again
  Theorem \ref{thm.SMPWeak} and Remark \ref{rem.Eigenvalue}, we obtain 
 \begin{equation} \label{eq:epositiveVal}
  \text{$u > 0$ a.e.\,in $\Omega$}.
 \end{equation}
 With \eqref{eq:epositiveVal} at hand, we now turn to prove
 that there exists $\alpha\geq 0$ such that
 $$e_1 = \alpha u.$$ 
 To this end we 
  observe that, on account of Theorem \ref{thm:uzeroglobalBd} and \eqref{glambda}
  in Remark \ref{rem.Eigenvalue},
 we know that $e_1,u\in L^\infty(\Omega)$. 
 Given any $\e>0$, we then define
 $$ 
 v_{\e}=\frac{u^p}{(e_1+\e)^{p-1}}.$$ 
 Since $v_{\e}\in \xp$ (as the same is true of both $e_1$ and $u$), we are entitled to use
  $v_{\e}$ as test function in  
 the problem solved by $e_1$. Thus, using again the notation
 $$
	J_p(t):=|t|^{p-2}t, \quad (t\in \R),
 $$
we obtain
\begin{equation}\label{eqe_1}
\begin{split}
 & \int_\Omega 
 |\nabla e_1|^{p-2}\langle \nabla e_1, \nabla v_{\e}\rangle dx \\
 & \qquad\qquad
 + \iint_{\mathbb{R}^{2n}} \frac{J_p((e_1+\e)(x)-(e_1+\e)(y))(v_{\e}(x)
  -v_{\e}(y))}{|x-y|^{n+2s}}\, dx\, dy \\
 & \qquad = 
  \lambda_1(\mathcal{L}_{p,s}+a)\int_{\Omega}e_1^{p-1}\,v_{\e}\, dx
   -\int_\Omega a(x)e_1^{p-1}\,v_{\e}\, dx.
\end{split}
\end{equation}
By the already recalled discrete Picone inequality, we find
$$
J_p((e_1+\e)(x)-(e_1+\e)(y))(v_\e(x)-v_\e(y))\leq |u(x)-u(y)|^p.
$$
Now, consider the function
\[
R(u,e_1+\e)= |\nabla u|^p - |\nabla e_1|^{p-2} 
 \langle \nabla e_1, \nabla {v_\e}\rangle.
\]
As a consequence of the nonlinear Picone identity by Allegretto - Huang in \cite{1} 
(see also \cite[p. 244]{MoMoPa}), we have that  $R(u,e_1+\e)\geq 0$. Then
\begin{equation}\label{pico}
|\nabla e_1|^{p-2} \langle \nabla e_1, \nabla v_\e\rangle \leq |\nabla u|^p.
\end{equation}
 Gathering these facts, 
 we can pass to the limit as $\e\to 0$ in \eqref{eqe_1}: 
  by applying the Fatou lemma in the left had side of \eqref{eqe_1} 
  and the Dominated Con\-ver\-gen\-ce theorem in the right hand side, we find
\begin{equation}\label{quasi}
\begin{split}
  & \int_\Omega |\nabla e_1|^{p-2}\Big\langle \nabla e_1, \nabla\Big(\frac{u^p}{e_1^{p-1}}
  \Big)\Big\rangle\, dx \\
  & \qquad\qquad
  + \iint_{\mathbb{R}^{2n}}\frac{J_p(e_1(x)-e_1(y))}{|x-y|^{n+2s}}
  \left(\frac{u^p(x)}{e_1^{p-1}(x)}-\dfrac{u^p(y)}{e_1^{p-1}(y)}\right) dx\, dy  \\
& \qquad \geq \lambda_1(\mathcal{L}_{p,s}+a)
 \int_{\Omega} u^p\, dx-\int_\Omega a(x)u^p\, dx \\
 & \qquad = 
  \int_\Omega |\nabla u|^p\, dx+\iint_{\R^{2n}}\frac{|u(x)-u(y)|^p}{|x-y|^{n+ps}}\, dx\, dy.
\end{split}
\end{equation}
 On the other hand, by using again inequality \eqref{pico},
  we have the estimate
\begin{equation}\label{quasi2}
\begin{split}
 &\int_\Omega |\nabla e_1|^{p-2}\Big\langle \nabla e_1, \nabla\Big(\frac{u^p}{e_1^{p-1}}
  \Big)\Big\rangle\, dx \\
  & \qquad\qquad 
  + \iint_{\mathbb{R}^{2n}}\frac{J_p(e_1(x)-e_1(y))}{|x-y|^{n+2s}}
  \left(\frac{u^p(x)}{e_1^{p-1}(x)}-\dfrac{u^p(y)}{e_1^{p-1}(y)}\right) dx\, dy 
   \\[0.1cm]
  & \qquad
  \leq  
  \int_\Omega |\nabla u|^p\, dx
  + \iint_{\R^{2n}}\frac{|u(x)-u(y)|^p}{|x-y|^{n+ps}}\, dx\, dy.
\end{split}
\end{equation}
 Hence, all the inequalities in \eqref{quasi} and \eqref{quasi2} are actually identities. In particular, 
 the discrete Picone inequality implies that
 $$
  \frac{e_1(x)}{e_1(y)}=\frac{u(x)}{u(y)} \mbox{ in $\mathbb{R}^{2n}$},
 $$
and so we can conclude that there exists $\alpha \geq 0$ such that 
$$\text{$e_1=\alpha u$ in $\R^n$}.$$
Now, suppose that $\lambda > \lambda_1(\mathcal{L}_{p,s}+a)$ is another eigenvalue of 
\eqref{eigen} with  
 associated $L^p-$normalized eigenfunction $u \in \xp$, and assume by contradiction that $u$ has constant 
 sign, say $u\geq0$. By Theorem \ref{thm.SMPWeak} we have $u>0$.

 Then, starting from the equation solved by $u$ and using
 $$
  \frac{e_1^p}{(u+\e)^{p-1}}
 $$
 as test function, by arguing exactly as for reaching \eqref{quasi}-\eqref{quasi2}, we get
$$
 \int_\Omega |\nabla e_1|^p\, dx + \iint_{\R^{2n}}\frac{|e_1(x)-e_1(y)|^p}{|x-y|^{n+ps}}\, dx\, dy 
 =\lambda 
-\int_\Omega a(x)  e_1^p\, dx.
$$
On the other hand, $e_1$ being a solution to \eqref{eigen} with $\lambda_1$, we have
$$
 \int_\Omega |\nabla e_1|^p\, dx
 + \iint_{\R^{2n}}\frac{|e_1(x)-e_1(y)|^p}{|x-y|^{n+ps}}\, dx\, dy 
 +\int_\Omega a(x) e_1^p\, dx=\lambda_1(\mathcal{L}_{p,s}+a).
$$
Since $\lambda>\lambda_1(\mathcal{L}_{p,s}+a)$, we get a contradiction, and thus $u$ must change sign.
\end{proof}

\section{Existence}\label{sec.Existence}
 In this last section we combine all the results established so far
 in order to give the proof of Theorem \ref{thm:Main}. Throughout what follows,
 we tacitly adopt all the notation introduced
 in Sections \ref{sec.NotPrel}-\ref{sec.Eigenvalue}: in particular,
 \begin{itemize}
  \item $\Omega\subseteq\R^n$ is a bounded open set with $C^1$ boundary;
  \item $f:\Omega\times\R\to\R$ satisfies (f1)--(f5);
  \item $a_0$ and $a_\infty$ are the functions defined in \eqref{a0ainfty};
  \item $\lambda_1(\LL_{p,s}-a_0)$ and $\lambda_1(\LL_{p,s}-a_\infty)$
  are defined in \eqref{eq:deflambdaBO}.
 \end{itemize}
 \begin{remark} \label{rem:assumptionf5bis}
  As already pointed out in the Introduction, the `sign
 assumption' (f5) is needed only to prove the uniqueness part
 of Theorem \ref{thm:Main}, since it allows us to invoke
 Theorem \ref{thm.uniqueness}; all the other results
 we are going to establish in this section actually hold under assumptions (f1)--(f4) solely.
 \end{remark}

 To begin with, we set
 $$F(x,u) = \int_{0}^{u}f(x,t)\, dt,$$ 
 and we consider the functional $E : \mathbb{X}_p(\Omega) \to \mathbb{R}$ defined as follows:
\begin{equation}\label{eq:defE}
  E(u):= \frac{1}{p} \int_{\Omega}|\nabla u|^p \, dx + 
  \frac{1}{p}\iint_{\mathbb{R}^{2n}}\!\!\!\frac{|u(x)-u(y)|^p}{|x-y|^{n+sp}}\, dx\, dy 
  - \int_{\Omega}F(x,u)\, dx.
\end{equation}
The functional $E$ is well-defined, differentiable and its critical 
points are weak solutions of problem \eqref{eq.pbDirSec2}.
\begin{proposition} \label{prop:sufficient}
Let $E$ be the functional defined in \eqref{eq:defE}, and assume that
\[
\lambda_1(\mathcal{L}_{p,s}-a_0)<0<\lambda_1 (\mathcal{L}_{p,s}-a_{\infty}).
\]
Then, the following hold:
\begin{itemize}
\item[(a)]$E$ is coercive on $\xp$.
\item[(b)]$E$ is weakly l.s.c.\,in $\xp$, so it has a minimum $v \in \xp$.
\item[(c)]There exists $\phi \in \mathbb{X}_p(\Omega)$ such that $E(\phi)<0$, so that
 $$\min_{u\in \xp}E(u)<0,$$ and $u=|v|$ is a solution to \eqref{eq.pbDirSec2}.
\end{itemize}
\begin{proof}
 (a)\,\,It is sufficient to note that, by its very definition,
\begin{equation}
E(u) \geq J(u):=  \dfrac{1}{p} \int_{\Omega}|\nabla u|^p \,  dx - \int_{\Omega}F(x,u)\, dx,
\end{equation}
\noindent for every $u \in \mathbb{X}_p(\Omega)$. Since we can identify $\mathbb{X}_p(\Omega)$ with 
$W^{1,p}_{0}(\Omega)$, the functional $J$ is precisely the one considered in \cite{DiazSaa} and therefore 
coercive, see \cite{BO} for the details in the linear case $p=2$. For completeness, we recall that the 
condition 
$$\lambda_1 (\mathcal{L}_{p,s}-a_{\infty})>0$$ 
is  used at this stage.
\medskip

(b)\,\,
 Let $u\in\xp$ be fixed, and let
 $\{u_n\}_n$ 
 be a sequence in $\xp$ which we\-a\-kly converges to $u$ as $n\to+\infty$. By (f3), we have
 $$|F(x,u)| \leq c_p (|u| + |u|^p);$$ 
 hence, by the Rellich-Kondrachev theorem we get
 $$\lim_{n\to +\infty}\int_{\Omega}F(x,u_n) \,  dx = \int_{\Omega}F(x,u)\, dx,$$
 which immediately implies the claim.
\medskip

(c)\,\,To prove this assertion, we can follow the argument originally presented in 
\cite{BO}. Since $\lambda_1(\mathcal{L}_{p,s}-a_0)<0$, there exists $\phi\in \mathbb{X}_p(\Omega)$ such 
that $\|\phi\|_{L^p(\Omega)} = 1$ and 
\begin{equation}\label{eq:sxp}
 \int_{\Omega}|\nabla \phi|^p \, dx + 
  \iint_{\mathbb{R}^{2n}} \dfrac{|\phi(x)-\phi(y)|^p}{|x-y|^{n+sp}} \, dx\, dy
 <\int_{\{\phi\neq 0\}} a_0\, |\phi|^p dx.
\end{equation}
 We then claim that it is not restrictive to assume that $\phi \geq 0$ 
 and $\phi\in L^{\infty}(\mathbb{R}^n)$. 
 In fact, since $||x|-|y||\leq |x-y|$ for every $x,y\in \R$, from \eqref{eq:sxp} we find
 \begin{align*} 
  & 
   \int_{\Omega}|\nabla |\phi||^p \, dx + 
   \iint_{\mathbb{R}^{2n}} \frac{||\phi(x)|-|\phi(y)||^p}{|x-y|^{n+sp}}\, dx\, dy \\
  & \qquad\leq 
     \int_{\Omega}|\nabla \phi|^p \, dx + 
      \iint_{\mathbb{R}^{2n}} \dfrac{|\phi(x)-\phi(y)|^p}{|x-y|^{n+sp}} \, dx\, dy
    <\int_{\{\phi\neq 0\}} a_0\, |\phi|^p dx,
\end{align*} 
 so that we can assume $\phi\geq 0$. As for the assumption $\phi\in L^\infty(\R^n)$, 
 we define 
 $$\phi_M=\min \{\phi,M\}\qquad (\text{for $M > 0$}).$$ 
 As usual, $\phi_M\in \xp$; moreover, since a direct computation gives 
  $$|\phi_M(x)-\phi_M(y)|\leq |\phi(x)-\phi(y)|,$$
  from \eqref{eq:sxp} we obtain 
\begin{align*}
 & \int_{\Omega}|\nabla \phi_M|^p \, dx + 
  \iint_{\mathbb{R}^{2n}} \frac{|\phi_M(x)-\phi_M(y)|^p}{|x-y|^{n+sp}} \, dx\, dy\\
 & \qquad \leq 
   \int_{\Omega}|\nabla \phi|^p \, dx + \iint_{\mathbb{R}^{2n}} 
    \frac{|\phi(x)-\phi(y)|^p}{|x-y|^{n+sp}} \, dx\, dy
   <\int_{\{\phi\neq 0\}} a_0\, |\phi|^p dx.
\end{align*}
 On the other hand, since $a_0$ is bounded from below (see \eqref{eq:azerobd}),
 we have
 $$
 \int_\Omega a_0\phi^p\leq \liminf_{M\to+\infty}\int_\Omega a_0\phi_M^p;$$ 
 as a consequence, can find $M>0$ large enough so that
 $$
   \int_{\Omega}|\nabla \phi_M|^p \, dx + \iint_{\mathbb{R}^{2n}}
    \frac{|\phi_M(x)-\phi_M(y)|^p}{|x-y|^{n+sp}} \, dx\, dy
    <\int_{\{\phi\neq 0\}} a_0\, |\phi_M|^p\, dx.
$$
 Summing up, by replacing $\phi$ with $|\phi_M|$, we can choose $\phi\geq0$ and bounded.
 \vspace*{0.05cm}
  
 Now, we have that
 $$
  \liminf_{u\to 0}\frac{F(x,u)}{u^p}\geq \frac{a_0(x)}{p}
 $$
and proceeding as in \cite[Proof of (15)]{BO} we get
\[
\liminf_{\varepsilon\to 0}\int_{\mathbb{R}^n}\frac{F(x,\varepsilon\phi)}{\varepsilon^p}\geq \frac{1}{p}\int_{\{\phi\neq 0\}} a_0\phi^p.
\]
Therefore using \eqref{eq:sxp} we conclude that
$$
 \int_{\Omega}|\nabla \phi|^p \, dx + \iint_{\mathbb{R}^{2n}}
  \frac{|\phi(x)-\phi(y)|^p}{|x-y|^{n+sp}} \, dx\, dy - 
  p\int_{\mathbb{R}^n}\frac{F(x,\varepsilon\phi)}{\varepsilon^p}< 0
$$
for any $\varepsilon>0$ small enough. Clearly, the latter can be rewritten as 
$$E(\varepsilon\phi)<0,$$ 
and this closes the proof.
\end{proof}
\end{proposition}

Concerning the ``necessity'' part, we start from the next result.
\begin{lemma} \label{lem:necessity}
Let $u\in \mathbb{X}_p(\Omega)$ be a solution of  
\eqref{eq.pbDirSec2}. Then
\[
\lambda_1(\mathcal{L}_{p,s}-a_0)<0.
\]
\end{lemma}
\begin{proof}
 On one hand, by the very definition of $\lambda_1(\LL_{p,s}-a_0)$, we have
$$\lambda_1(\mathcal{L}_{p,s}-a_0)\leq 
\frac{\mathcal{Q}_{p,s}(u)-\int_{\{u\neq 0\}}a_0\,|u|^p}{\|u\|_{L^p(\Omega)}},
$$
 where $\mathcal{Q}_{p,s}$ is as in \eqref{eq:defQform}. On the other hand, since $u$ solves \eqref{eq.pbDirSec2}, we have that
\[
\int_{\Omega}|\nabla u|^p\, dx + 
\iint_{\R^{2n}}\frac{|u(x)-u(y)|^p}{|x-y|^{n+ps}} =  \mathcal{Q}_{p,s}(u) = 
\int_\Omega f(x,u)u\, dx.
\]
By the strong maximum principle, $u>0$ in $\Omega$. Therefore, by definition of $a_0$ and by assumption (f4), we get
that
\[
\frac{f(x,u)}{u^{p-1}}<a_0(x) \mbox{ a.e. in }\Omega,
\]
so that
\[
\int_\Omega f(x,u)u\, dx <\int_\Omega a_0u^p dx
\]
and the conclusion follows.
\end{proof}

Although up to now we have been able to treat the general case, we are now led to focus on the semilinear case.
\begin{proposition}\label{PropAinfPos}
Assume that $ p =2$, and let $u\in \mathbb{X}_2(\Omega)$ be a nonnegative solution of 
problem \eqref{eq.pbDirSec2}. Then
$$\lambda_1 (\mathcal{L}_{2,s}-a_{\infty}) >0.$$
\begin{proof}
 First of all we observe that, in view of Theorem \ref{thm:uzeroglobalBd}, we know that
 $$u\in L^\infty(\Omega).$$ Hence, 
 as in \cite{BO}, we define the bounded and indefinite weight
$$\overline{a}(x) := \dfrac{f(x,\|u\|_{L^{\infty}(\Omega)}+1)}{\|u\|_{L^{\infty}(\Omega)}+1}.$$
Notice that $\overline{a} \in L^{\infty}(\Omega)$ by (f2).
Then, we consider the auxiliary eigenvalue problem
\begin{equation}
\left\{ \begin{array}{rl} 
             \mathcal{L}_{2,s}\psi - \overline{a}(x) \psi = \mu \psi & \textrm{in }\Omega,\\
             \psi \geq 0 & \textrm{in } \Omega,\\
             \psi = 0 &\textrm{in } \mathbb{R}^{n}\setminus \Omega.
\end{array}\right.
\end{equation}
By Proposition  \ref{Proposition3} we get the existence of a principal eigenvalue with associated bounded and nonnegative eigenfuction $\psi \in \tilde{H}$. We can therefore use such a $\psi$ as test function for \eqref{eq.pbDirSec2}, finding
\begin{equation}
\int_{\Omega}u\psi \, (\overline{a}+\mu) \, dx = \int_{\Omega} f(x,u)\psi \, dx.
\end{equation}
Clearly, 
$$\int_{\Omega}u\psi \, (\overline{a}+\mu) \, dx = \int_{\Omega \cap \{u>0\}}u\psi \, (\overline{a}+\mu) \, dx,$$
\noindent and on $\Omega \cap \{u>0\}$ we can exploit condition (f4), which yields
$$\int_{\Omega \cap \{u>0\}} f(x,u)\psi \, dx > \int_{\Omega \cap \{u>0\}} \overline{a}(x) u\psi \, dx = \int_{\Omega} \overline{a}(x) u\psi \, dx.$$
Therefore, we find that
$$\mu \int_{\Omega}u\psi \, dx >0,$$
\noindent and then conclude as in \cite{BO}.
\end{proof}
\end{proposition}
By combining the results in this section, we can finally prove Theorem \ref{thm:Main}.
\begin{proof}[Proof of Theorem $\ref{thm:Main}$]
 As for the uniqueness, it is a consequence of 
 Lemma \ref{lem:necessity},
 together with Theorem \ref{thm.uniqueness}.
 Moreover, the strict positivity is contained in Corollary \ref{cor.posSol}.
 
 The existence part of assertion (1) is exactly the content of Proposition \ref{prop:sufficient}.
 As for assertion (2), it follows from
 Lemma \ref{lem:necessity} and Pro\-po\-si\-tion \ref{PropAinfPos}.
\end{proof}
 
\end{document}